\documentclass[a4paper,dvipsnames]{amsart}

\usepackage[english]{babel}
\usepackage[T1]{fontenc}

\usepackage{amsmath}
\usepackage{amssymb}
\usepackage{amsaddr}
\usepackage{enumitem}
\setlist[enumerate]{label=\textnormal{(\arabic*)}}
\usepackage{amsthm}
\usepackage[dvipsnames]{xcolor}
\usepackage[
pagebackref,
hypertexnames=false,
colorlinks,
citecolor=BrickRed,
linkcolor=MidnightBlue,
urlcolor=BrickRed,
pdftitle={Empirical bounds for commuting dilations of free unitaries},
]{hyperref}
\usepackage[capitalise]{cleveref}
\crefname{condition}{condition}{conditions}
\crefname{formula}{formula}{formulas}
\crefname{equation}{}{}
\Crefname{equation}{Equation}{Equations}

\usepackage{nicefrac}

\usepackage{tikz}
\usepackage{tikz-cd}
\usetikzlibrary{plotmarks}

\usepackage{dsfont}
\usepackage{mathtools}
\usepackage{scalerel}
\usepackage{stackengine}

\usepackage{etoolbox}
\apptocmd{\sloppy}{\hbadness 10000\relax}{}{}
\apptocmd{\sloppy}{\vbadness 10000\relax}{}{}

\usepackage{mathtools}

\usepackage[a4paper,top=3cm,bottom=2cm,left=3cm,right=3cm,marginparwidth=1.75cm]{geometry}

\usepackage{comment}  

\usepackage[english]{babel}
\usepackage{colonequals} 

\usepackage{hyphenat}

\usepackage{academicons}
\definecolor{orcidcol}{HTML}{A6CE39}

\usepackage{faktor}

\usepackage{graphicx}
\usepackage{subcaption}

\colorlet{Malte}{OliveGreen}
\colorlet{Philipp}{blue}

\newcommand{\N}{\mathbb N} 

\newcommand{\orcidlogo}{{\includegraphics[width=\fontcharht\font`l]{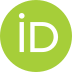}}}

\newcommand{\ie}{i.e.\ }

\newcommand{\be}{\begin{equation}}
\newcommand{\ee}{\end{equation}}
\newcommand{\bes}{\begin{equation*}}
\newcommand{\ees}{\end{equation*}}

\newcommand{\fD}{\mathfrak{D}}

\newcommand{\cA}{\mathcal{A}}

\newcommand{\cH}{\mathcal{H}}

\newcommand{\cK}{\mathcal{K}}

\newcommand{\cU}{\mathcal{U}}

\newcommand{\cW}{\mathcal{W}}

\newcommand{\bB}{\mathbb{B}}
\newcommand{\bC}{\mathbb{C}}
\newcommand{\bD}{\mathbb{D}}

\newcommand{\bF}{\mathbb{F}}

\newcommand{\bN}{\mathbb{N}}

\newcommand{\bR}{\mathbb{R}}

\newcommand{\bT}{\mathbb{T}}
\newcommand{\bZ}{\mathbb{Z}}

\newcommand{\ol}{\overline}

\newcommand{\conv}{\operatorname{conv}}

\newcommand{\distance}{\operatorname{d}}

\newtheorem{theorem}{Theorem}[section]

\newtheorem{corollary}[theorem]{Corollary}
\newtheorem{proposition}[theorem]{Proposition}

\theoremstyle{definition} 
\newtheorem{definition}[theorem]{Definition}
\newtheorem{remark}[theorem]{Remark}

\theoremstyle{remark}

\title[Empirical bounds for commuting dilations of free unitaries]{Empirical bounds for commuting dilations of free unitaries and the universal commuting dilation constant}

\author{
  Malte Gerhold$^{\MakeLowercase{a}}$ \href{https://orcid.org/0000-0003-4029-1108}{$\orcidlogo$},
  Marcel Scherer$^{\MakeLowercase{b}}$ \and
  Orr Moshe Shalit$^{\MakeLowercase{b,c}}$ \href{https://orcid.org/0000-0002-5390-6928}{$\orcidlogo$}
}

\thanks{M.G.\ was supported by the German Research Foundation (DFG) grant no.\ 397960675 and by the Alfried Krupp Wissenschaftskolleg Greifswald.}
\thanks{O.M.S.\ was supported by ISF Grant no. 431/20}

\address{$^{a}$ Institute of Mathematics and Computer Science, University of Greifswald}

\address{$^{b}$ Faculty of Mathematics, Technion Israel Institute of Technology}

\address{$^{c}$ Helen Diller Quantum Center, Technion Israel Institute of Technology}


\makeatletter
\hypersetup{
pdftitle={\@title},
pdfauthor={Malte Gerhold, Marcel Scherer and Orr Moshe Shalit},
}
\makeatother
\date{}

\begin{document}
\begin{abstract}
For a tuple $T$ of Hilbert space operators, the \emph{commuting dilation constant} is the smallest number $c$ such that the operators of $T$ are a simultaneous compression of commuting normal operators of norm at most $c$. We present numerical experiments giving a strong indication that the commuting dilation constant of a pair of independent random $N{\times}N$ unitary matrices converges to $\sqrt2$ as $N \to \infty$ almost surely. Under the assumption that this is the case, we prove that the commuting dilation constant of an arbitrary pair of contractions is strictly smaller than $2$. Our experiments are based on a simple algorithm that we introduce for the purpose of computing dilation constants between tuples of matrices.
\end{abstract}

\maketitle

\section{Introduction}

In recent years, dilation theoretic techniques and the framework of matrix convex sets have found new and significant applications to quantum information \cite{BlNe18,BlNe20,BNS23}, optimization and control \cite{HKMS19}, mathematical physics \cite{ALPP21,gerhold2023bounded}, the theory of operator systems and operator algebras \cite{DP22,FNT17,PaPa21}, and more \cite{CDN20,CDNV23,EPS25,HKM17}. 
The goal of this paper is to advance our knowledge regarding an open problem in this area: determining the {\em universal commuting dilation constant} $C_2$, \ie the smallest number $c$ such that every pair of contractions is the simultaneous compression of commuting normal operators of norm at most $c$. By combining rigorous limit theorems with numerical experimentation, we provide evidence that $C_2 \leq 2 \sqrt{\frac{2}{3}} < 2$. 

\subsection{Preliminaries and background}
Given two $d$-tuples of operators $A = (A_1, \ldots, A_d)$ and $B = (B_1, \ldots, B_d)$ in $B(\cH)^d$ on the same Hilbert space $\cH$, we define the distance
\be\label{eq:norm_dist}
\|A - B\| = \max_{1\leq i \leq d} \|A_i - B_i\|, 
\ee
induced by the norm $\|A\| =\max_i \|A_i\|$. We shall use the same notation for tuples of matrices. 
In the space $M_n^d$ consisting of all $d$-tuples of complex $n \times n$ matrices, it is convenient to define 
\[
\fD_d(n) = \{A \in M_n^d : \|A\| < 1\}
\]
and $\fD_d = \sqcup_{n=1}^\infty \fD_d(n)$. 
The \emph{matrix range} of an operator tuple $A \in B(\cH)^d$ is given by the disjoint union $\cW(A) = \sqcup_{n=1}^\infty \cW_n(A)$ where
\[
\cW_n(A) = \bigl\{\phi(A):= \bigl(\phi(A_1), \ldots, \phi(A_d)\bigr) :\phi\in \operatorname{UCP}\bigl(C^*(A), M_n\bigr)\bigr\}. 
\]
If $B \in B(\cK)^d$ is another operator tuple, we define the {\em matrix range distance} by 
\[
\distance_{\mathrm{mr}}(A,B):=\distance_{\mathrm H}\bigl(\cW(A),\cW(B)\bigr) := \sup_n \distance_{\mathrm H}\bigl(\cW_n(A),\cW_n(B)\bigr) ,
\]
the supremum over the levelwise Hausdorff distances\footnote{Recall that the Hausdorff distance $\distance_{\mathrm H}(E,F)$ between to sets $E,F \subset M_n^d$ determined by the norm \cref*{eq:norm_dist} is the infimum over all $r>0$ such that $E \subset F + r\cdot \fD_d(n)$ and $F \subset E + r \cdot \fD_d(n)$.} $\distance_{\mathrm H}\bigl(\cW_n(A),\cW_n(B)\bigr)$ in $M_n^d$ induced by the norm \cref*{eq:norm_dist}. 
We define the \emph{one-sided matrix range distance} between $A \in B(\cH)^d$ and $B \in B(\cK)^d$:
\[
  \distance_{\mathrm{mr}}(A\to B):= \sup_n \sup_{X\in \cW_n(A)}\mathop{\inf\vphantom{sup}}_{Y\in \cW_n(B)} \|X-Y\| .
\]
The one sided matrix range distance $\distance_{\mathrm{mr}}(B\to A)$ was discussed before in \cite[Section 5]{DDSS17} where it was denoted by $\delta_{\cW(A)}(\cW(B))$. The {matrix range distance} is then given by 
\[
\distance_{\mathrm{mr}}(A,B)= \max \bigl(\distance_{\mathrm{mr}}(A\to B),\distance_{\mathrm{mr}}(A\to B)\bigr) .
\]
The matrix range distance is a metric only when considered on classes of certain {\em rigid} tuples, for example on the set $\mathcal U(d)$ of unitary $d$-tuples. 
For general operator tuples it does not determine the tuple up to $*$-isomorphism, since by \cite[Proposition 5.5]{DDSS17}, $\distance_{\mathrm{mr}}(A,B) = 0$ if and only if there is unital completely isometric map sending $A_i$ to $B_i$ for $i=1, \ldots, d$ (thus the matrix range is a complete invariant of the operator system generated by a tuple, and it is known that completely isometric operator systems may generate non-isomorphic C*-algebras). 
In fact, \cite[Proposition 5.5]{DDSS17} gives more generally that $\cW(B) \subseteq \cW(A)$ if and only if there exists a unital completely positive (UCP) map $\phi$ such that $B_i = \phi(A_i)$ for $i = 1, \ldots, d$. 

Another measure of difference between tuples is given by the {\em dilation distance} \cite{GPSS21} defined as follows. 
First, we define the dilation constant $c(u,v)$. 
Given two unitary tuples $u,v$ and a real number $c>0$, we write $u \prec cv$ if there exist two Hilbert spaces $\cH \subseteq \cK$ and two operator tuples $U \in B(\cH)^d$ and $V \in B(\cK)^d$, such that $u \sim U$, $v \sim V$ and 
\[
U = P_\cH cV \big|_\cH;
\]
here and below, the notation $u \sim U$ means that there is a $*$-isomorphism $\pi\colon C^*(u) \to C^*(U)$ such that $U_i = \pi(u_i)$ for all $i=1, \ldots, d$. 
By Stinespring's theorem, $u \prec cv$ if and only if there is a UCP map from the operator system generated by $v$ to the operator system generated by $u$, that maps $cv$ to $u$. 
For $u,v \in \cU(d)$, we define the \emph{dilation constant} $c(u,v)$ to be  
\[
c(u,v) = \inf\{c :  u \prec cv\}. 
\]
The {\em dilation distance} is then defined by 
\[
\distance_{\mathrm{D}}(u,v):= \log \max\bigl\{c(u,v), c(v,u)\bigr\} .
\]
Various aspects of the dilation distance were studied in \cite{GPSS21}. For the {\em Haagerup-Rørdam distance}
\[
\distance_{\mathrm{HR}}(u,v):=\inf\left\{\left\|u'-v'\right\|\colon u', v' \in B(\cH)^d,  u \sim u' \textrm{ and } v \sim v'\right\}, 
\]
in \cite[Theorem 2.6]{GPSS21} it was proved that 
\[
\distance_{\mathrm{HR}}(u,v) \leq K \distance_{\mathrm{D}}(u,v)^{1/2}
\]
where $K$ is a universal constant (for further developments see \cite{gerhold2023bounded,gerhold2024dilation,GLL26}).

\subsection{Overview of this paper}\label{sec:overview}

Let $u_{\mathrm u}$ denote the universal $d$-tuple of unitaries, \ie the canonical generators of the full group C*-algebra $C^*(\bF_d)$ of the free group $\bF_d$. 
Let $u_0$ denote the universal $d$-tuple of commuting unitaries, \ie the canonical generators of the algebra $C(\bT^d)$ of continuous functions on the torus $\bT^d$. 
The constant $C_d := c(u_{\mathrm u}, u_0)$ is the minimal constant such that every $d$-tuple $A = (A_1, \ldots, A_d)$ of contractions can be dilated to a $d$-tuple $N = (N_1, \ldots, N_d)$ of commuting normal elements such that $\|N_i\| \leq C_d$ for all $i=1, \ldots, d$.
Note that here and below we are keeping the $d$ implicit in the notation; there are different tuples $u_{\mathrm u}, u_0$  and different $c(u_{\mathrm u}, u_0)$ for every $d$. 

For some time now it has been an open problem to determine the precise value of $C_d$ for all $d$. 
It is not hard to show that $C_d \leq d$ \cite{DDSS17, HKMS19}. 
This was improved to $C_d \leq \max\{d, 2 \sqrt{d}\}$ in \cite{PSS18} and to $C_d \leq \sqrt{2d}$ in \cite{Pas19}. 
The lower bound $C_d \geq \sqrt{d}$ follows from the selfadjoint case obtained in \cite{PSS18}. 
For general $d$ the best bounds are $\sqrt{d} \leq C_d \leq \sqrt{2d}$. 
However, for $d = 2$ and $d = 3$, it was shown in \cite{GSh20} and \cite{GPSS21}, respectively, that $C_d > \sqrt{d}$. 
Thus, the lower bound for $C_d$ is not tight. In this paper we provide evidence that the upper bound is not tight, either.
We focus on the case $d = 2$, in which the currently known best upper bound $C_2 \leq \sqrt{2 \cdot 2}$ coincides with the easy bound $C_d \leq d$. 

Let $C_2(n)$ denote the constant such that every pair of commuting $n \times n$ matrices has a dilation to $C_2(n)$ times a pair of commuting unitaries. 
By an explicit dilation construction, we show in the appendix that $C_2(n)\le\sqrt{2+2\sin(\tfrac{\pi}{2}(1-\tfrac{1}{2n}))} < 2$; see \Cref{cor:C2n}. This bound is likely to be very crude, and in any case it tends to $2$ as $n \to \infty$, so cannot be used to show that $C_2 < 2$. 
We can obtain less rigorous but more convincing results by other means.

In \cite{GSh21}, we studied matrix ranges of random matrix tuples. 
We showed that if $T^{(N)} = (T_1^{(N)}, \ldots, T_d^{(N)})$ is a random matrix ensemble and $t = (t_1, \ldots, t_d)$ a $d$-tuple of operators such that
\be\label{eq:SAF}
\lim_{N\to\infty}\|p(T^{(N)})\|=\|p(t)\|
\ee
almost surely, for every $*$-polynomial $p \in \bC\langle z, z^* \rangle$, then for all $n \in \bN$, 
\[
\distance_{\mathrm H}(\cW_n(T^{(N)}),\cW_n(t)) \xrightarrow{N \to \infty} 0,
\]
almost surely (see \cite[Theorem 3.1]{GSh21}). When combined with known results on strong convergence of matrix ensembles, this result implies convergence of matrix ranges of these matrix ensembles. For example, if $X^{(N)}$ are tuples of independent matrices from a Wigner ensemble, then for all $n$,
\[
\lim_{N \to \infty} \distance_{\mathrm H}(\cW_n(X^{(N)}),\cW_n(s)) = 0 \,\, , \,\, a.s.
\]
where $s = (s_1, \ldots, s_d)$ is a tuple of free semicirculars. 

More interesting for our purposes are ensembles of unitaries. Recall that a \emph{free tuple of Haar unitaries} is a $d$-tuple $u_{\mathrm f} = (u_{\mathrm{f} 1}, \ldots, u_{\mathrm{f}d})$ of unitaries in a C*-probability space $(\cA, \tau)$ which are freely independent and satisfy $\tau(u_{\mathrm{f}i}^k) = 0$ for all $i$ and all $k \in \bZ \setminus \{0\}$.
By \cite[Theorem 3.4]{GSh21}, if $U^{(N)} = (U^{(N)}_1, \ldots, U^{(N)}_d)$ are $d$ independent $N \times N$ unitaries sampled from the Haar measure on $\cU_N$, and $u_{\mathrm f} = (u_{\mathrm{f} 1}, \ldots, u_{\mathrm{f}d})$ is a free Haar unitary tuple, then the matrix range $\cW(U^{(N)})$ converges almost surely levelwise in the Hausdorff metric to $\cW(u_{\mathrm f})$, that is, for all $n$,
\be\label{eq:UNu}
\lim_{N \to \infty} \distance_{\mathrm H}\left(\cW_n(U^{(N)}),\cW_n(u_{\mathrm f})\right) = 0 \,\, , \,\, a.s.
\ee
In \Cref{thm:limit_c,cor:c_converge} below, we shall prove that a limit such as \cref*{eq:UNu} implies that 
\be\label{eq:limiting}
c(u_{\mathrm f}, u_0) \leq \liminf_{N \to \infty} c(U^{(N)},u_0) .
\ee 
On the other hand, by \cite[Corollary 3.8]{GPSS21}, we have the exact value
\[
c(u_{\mathrm u}, u_{\mathrm f}) = \frac{d}{\sqrt{2d-1}}. 
\]
We also have the trivial ``triangle inequality'' $c(u_{\mathrm u}, u_0) \leq c(u_{\mathrm u}, u_{\mathrm f}) c(u_{\mathrm f}, u_0)$. Putting everything together, we get 
\be\label{eq:tri_ineq}
C_d = c(u_{\mathrm u}, u_0) \leq c(u_{\mathrm u}, u_{\mathrm f}) c(u_{\mathrm f}, u_0) \leq \frac{d}{\sqrt{2d-1}} c(u_{\mathrm f}, u_0). 
\ee
In \Cref{sec:experimental_re}, we document empirical evidence suggesting that when $d = 2$ we have the limit 
\[
\lim_{N \to \infty} c(U^{(N)},u_0) = \sqrt{2}. 
\]
The evidence is gathered in an experimental setup described in \Cref{subsec:approx}, which is powered by an algorithm that we propose in \Cref{subsec:algo} for computing $c(\mathsf{U},\mathsf{N})$ for a given pairs of matrices $\mathsf{U}$ and $\mathsf{N}$.
Combining the empirical result with the rigorous relations \cref*{eq:limiting,eq:tri_ineq}, we semi-rigorously derive the inequality 
\[
C_2 \leq 2\sqrt{\frac{2}{3}} < 2 , 
\]
which we conjecture to hold true. 

It is worth recalling that by Theorems 3.9 and 3.10 in \cite{GPSS21}, we have the estimate
\be\label{eq:free_bounds}
2 \sqrt{1 - \frac{1}{d}} \leq c(u_{\mathrm f}, u_0) \leq 2 \sqrt{1 - \frac{1}{2d}}. 
\ee
Plugging the upper bound in \Cref{eq:tri_ineq} one recovers the upper bound $C_d \leq \sqrt{2d}$, which was obtained in \cite{Pas19}; in the case $d = 2$ this is equal to the trivial bound $C_2 = 2$. 
On the other hand, in the case $d = 2$, the lower bound gives $\sqrt{2}$ which is what the data suggests. 

\section{Levelwise convergence and matrix range distance}

\begin{theorem}\label{thm:limit_c}
  Let $(\xi^{(N)})_{N\in\N}$ and $(\eta^{(N)})_{N\in\N}$ be sequences of operator $d$-tuples whose matrix ranges converge levelwise to the matrix ranges of operator $d$-tuples $\xi^{(\infty)}$ and $\eta^{(\infty)}$, respectively. Then
  \[  \mathrm{d_{mr}}(\xi^{(\infty)}\to\eta^{(\infty)})\leq \liminf_{N,M\to\infty} \mathrm{d_{mr}}(\xi^{(N)}\to\eta^{(M)})\]
  and
  \[  c(\xi^{(\infty)},\eta^{(\infty)})\leq \liminf_{N,M\to\infty} c(\xi^{(N)},\eta^{(M)}).\]
\end{theorem}

\begin{proof}
  Let 
  \[
  r> \liminf_{N,M\to\infty} \mathrm{d_{mr}}(\xi^{(N)}\to\eta^{(M)}). 
  \]
  Then there are sequences of natural numbers $N_k,M_k\xrightarrow{k\to\infty}\infty$ such that 
  \[
  \cW_n(\xi^{(N_k)})\subset \cW_n(\eta^{(M_k)})+ r{\fD_d(n)}
  \]
  for all $n,k\in\N$. It follows from the assumption of levelwise convergence that 
  \[
  \cW_n(\xi^{(\infty)})\subset \cW_n(\eta^{(\infty)}) + r\overline{\fD_d(n)}
  \] 
  for all $n\in\N$, \ie $r\geq \mathrm{d_{mr}}(\xi^{(\infty)}\to\eta^{(\infty)})$. Therefore, the claimed inequality must hold.
  
 The second inequality is proved analogously: Let 
 \[
 r> \liminf_{N,M\to\infty} c(\xi^{(N)},\eta^{(M)}).
 \] 
 Then there are sequences $N_k,M_k\xrightarrow{k\to\infty}\infty$ and $r_k < r$ such that $\xi^{(N_k)}\prec r_k \eta^{(M_k)}$ for all $k\in\N$ or, equivalently (by \cite[Proposition 5.5]{DDSS17}), $\cW_n(\xi^{(N_k)})\subset r_k \cW_n(\eta^{(M_k)})$ for all $n,k\in\N$. 
 By passing to a subsequence, we may assume that $r_k \to r_\infty \leq r$. 
 It follows from the assumption of levelwise convergence that $\cW_n(\xi^{(\infty)})\subset r_\infty \cW_n(\eta^{(\infty)})$ for all $n\in\N$, \ie $\xi^{(\infty)}\prec r_\infty \eta^{(\infty)}$, which implies to $c(\xi^{(\infty)},\eta^{(\infty)}) \leq r_\infty \leq r$. 
 Therefore, the claimed inequality must hold.   
\end{proof}

\begin{remark}
Consider a sequence $\xi^{(N)}$ of tuples such that $\cW(\xi^{(N)})$ converges levelwise, but not uniformly to $\cW(\xi^{(\infty)})$ (for the existence of examples of such tuples, see \cite{PaPa21}), and let $\eta^{(\infty)} = \eta^{(N)} = \xi^{(\infty)}$. 
Then by passing to a subsequence, we find that 
\[
\lim_{N,M\to\infty} \mathrm{d_{mr}}(\xi^{(N)},\eta^{(M)}) > 0 = \mathrm{d_{mr}}(\xi^{(\infty)},\eta^{(\infty)}) , 
\]
showing that weak inequality cannot, in general, be replaced by an equality in \Cref{thm:limit_c}. 
Similarly, one can show that the limit inferior cannot be replaced by a limit; indeed, with $\xi^{(N)}$ as before, the double sequence $\mathrm{d_{mr}}(\xi^{(N)},\xi^{(M)})$ cannot converge. 
\end{remark}

\begin{definition}
We say that an ensemble $T^{(N)}$ of random matrices has the strong asymptotic freeness property if \Cref{eq:SAF} holds for a freely independent tuple $t$. 
\end{definition}
\begin{corollary}\label{cor:c_converge}
  If a random matrix ensemble $T^{(N)}$ converges in distribution to $u_{\mathrm f}$ and has the strong asymptotic freeness property, then with probability $1$
  \be\label{eq:c_lt_liminf}
  c(u_{\mathrm f}, u_0)\leq\liminf c(T^{(N)},u_0).
  \ee
  In particular, this holds for $T^{(N)}$ a $d$-tuple of $N{\times} N$
  \begin{itemize}
  \item independent Haar distributed unitaries,
  \item matrices with $T^{(N)}_1$ deterministic unitaries whose empirical eigenvalue distribution tends to the Haar measure on $\mathbb T$, and $T^{(N)}_2,\ldots,T^{(N)}_d$ independent Haar distributed unitaries,
  \item independent uniformly distributed permutation matrices.
  \end{itemize}   
\end{corollary}

\begin{proof}
  By the strong asymptotic freeness property and \cite[Theorem 3.1]{GSh21},  $\cW(T^{(N)})$ converges levelwise to $\cW(u_{\mathrm f})$, so  Equation \eqref{eq:c_lt_liminf} follows from Theorem \ref{thm:limit_c}.
  Strong asymptotic freeness for the first two named cases is proved in \cite{CM14}. Strong asymptotic freeness of permutation matrices in standard representation (\ie acting on the orthogonal complement of the Perron-Frobenius Eigenvector) is shown in \cite{BoCo19}. This is good enough, because $T=\pi_{\mathrm{std}}(T)\oplus \pi_{\mathrm{trivial}}(T)$ and the $\pi_{\mathrm{trivial}}(T_i^{(N)})=1$ commute.    
\end{proof}

\begin{remark}
\Cref{cor:c_converge} can also be shown to hold without recourse to matrix ranges, using the fact that by the implication $(i) \Rightarrow (ii)$ in \cite[Theorem 2.3]{GSh21}, if \Cref{eq:SAF} holds for all $*$-polynomials, then it also holds for all matrix valued $*$-polynomials. 
We omit the details. 
\end{remark}

\begin{proposition}\label{prop:limsup}
  If $\lim \mathrm{d_{mr}}(\xi^{(N)}\to\xi^{(\infty)})=0$ and $\lim \mathrm{d_{mr}}(\eta^{(\infty)}\to\eta^{(N)})=0$, then
   \[\limsup \mathrm{d_{mr}}(\xi^{(N)}\to\eta^{(M)}) \leq \mathrm{d_{mr}}(\xi^{(\infty)}\to\eta^{(\infty)})\]
\end{proposition}

\begin{proof}
  By assumption,
  \begin{multline*}
    \mathrm{d_{mr}}(\xi^{(N)}\to\eta^{(M)}) \leq  \mathrm{d_{mr}}(\xi^{(N)}\to\xi^{(\infty)}) +  \mathrm{d_{mr}}(\xi^{(\infty)}\to\eta^{(\infty)}) +  \mathrm{d_{mr}}(\eta^{(\infty)}\to\eta^{(M)}) \\\xrightarrow{M,N\to \infty}  \mathrm{d_{mr}}(\xi^{(\infty)}\to\eta^{(\infty)})
  \end{multline*}
  and the claim follows.
\end{proof}

The respective condition on the $\xi^{(N)}$ or $\eta^{(N)}$ in \Cref{prop:limsup} is trivially fulfilled if the sequence is constant (which can still be interesting, as long as one of the two is non-constant).  But there are also properties of the limits $\xi^{(\infty)}$ and $\eta^{(\infty)}$ which guarantee the respective condition for arbitrary sequences which converge levelwise. 

\begin{proposition}
  For an operator tuple $X\in B(\cH)^d$, we write $\cW^{k\text{-$\max$}}(X)$ and $\cW^{k\text{-$\min$}}(X)$ for the maximal and minimal matrix convex set containing $\cW_k(X)$, respectively. Let $\xi^{(N)}, \xi^{(\infty)}, \eta^{(N)}, \eta^{(\infty)}$ be operator $d$-tuples. 
  \begin{enumerate}
  \item If $\cW(\xi^{(\infty)})$ is the levelwise limit of $\cW(\xi^{(N)})$ and the uniform limit of  $\cW^{k\text{-$\max$}}(\xi^{(\infty)})$, \ie if
  \[
     \forall n .  \lim_{N\to \infty} \distance_{\mathrm H}\left(\cW_n(\xi^{(N)}),\cW_n(\xi^{(\infty)})\right) = 0 
     \,\, \textrm{ and } \,\,
     \lim_{k\to\infty} \distance_{\mathrm{mr}} \left(\cW^{k\text{-$\max$}}(\xi^{(\infty)}), \cW(\xi^{(\infty)})\right) = 0 , 
  \]
  then
  \[
  \lim_{N \to \infty} \mathrm{d_{mr}}(\xi^{(N)}\to\xi^{(\infty)}) = 0 .
  \]
  \item If $\cW(\eta^{(\infty)})$ is the levelwise limit of $\cW(\eta^{(N)})$ and the uniform limit of  $\cW^{k\text{-$\min$}}(\eta^{(\infty)})$, \ie if
  \[
     \forall n .  \lim_{N\to \infty} \distance_{\mathrm H}\left(\cW_n(\eta^{(N)}),\cW_n(\eta^{(\infty)})\right) = 0
     \,\, \textrm{ and } \,\, 
     \lim_{k\to\infty} \distance_{\mathrm{mr}} \left(\cW^{k\text{-$\min$}}(\eta^{(\infty)}), \cW(\eta^{(\infty)})\right) = 0 , 
  \]
  then 
  \[
  \lim_{N \to \infty} \mathrm{d_{mr}}(\eta^{(\infty)}\to\eta^{(N)}) = 0 .
  \]
  
  \end{enumerate}
\end{proposition}

\begin{proof}
  For every $\varepsilon>0$ and every $k$, there exists an $N_0$ such that for all $N>N_0$ we have
  \[\cW_k(\xi^{(N)})\subseteq \cW_k(\xi^{(\infty)}) +  \varepsilon {\fD_d}\]
  or, equivalently,
  \[\cW(\xi^{(N)})\subseteq \cW^{k\text{-$\max$}}(\xi^{(\infty)}) + \varepsilon {\fD_d}.\]
  Now assume that $\lim_{k\to\infty}\cW^{k\text{-$\max$}}(\xi^{(\infty)})=\cW(\xi^{(\infty)})$.
  Given $\varepsilon$, we can choose $k$ such that
  \[\cW^{k\text{-$\max$}}(\xi^{(\infty)})\subseteq \cW(\xi^{(\infty)}) +\varepsilon {\fD_d}\]
  and, therefore, we have for all $N>N_0$ 
  \[\cW(\xi^{(N)})\subseteq \cW^{k\text{-$\max$}}(\xi^{(\infty)}) + \varepsilon {\fD_d}\subseteq  \cW(\xi^{(\infty)}) + 2 \varepsilon {\fD_d}.\]
  This shows that $\mathrm{d_{mr}}(\xi^{(N)}\to\xi^{(\infty)})$ converges to $0$. The second claim is proved analogously, noting that
  $\cW_k(\eta^{(\infty)})\subset \cW_k(\eta^{(N)})  +\varepsilon {\fD_d}$ if and only if $\cW^{k,\min}(\eta^{(\infty)})\subset \cW(\eta^{(N)})+\varepsilon {\fD_d}$.
\end{proof}

The uniform convergence of $\cW^{k\text{-$\max$}}(X)$ and $\cW^{k\text{-$\min$}}(X)$ to $\cW(X)$ was characterized by Passer and Paulsen in operator system terms (lifting property and 1-exactness) in \cite[Theorems 3.3 and 3.7]{PaPa21}. 

We close this section with a lengthy remark on the selfadjoint case. 
\begin{remark}\label{rem:selfadjoint}
Consider the dilation constants $\alpha_d, \beta_d, \gamma_d, \delta_d$ defined as follows:
\begin{enumerate}
\item $\alpha_d$ is the smallest constant such that every $d$ selfadjoint contractions can be dilated to a $d$-tuple of commuting selfadjoint operators of norm at most $\alpha_d$. 
\item $\beta_d$ is the smallest constant such that every $d$ selfadjoint contractions can be dilated to a $d$-tuple of commuting selfadjoint operators of the form $\beta_d \times$unitary.
\item $\gamma_d$ is the smallest constant such that every $d$ selfadjoint contractions can be dilated to a $d$-tuple of commuting normal operators of norm at most $\gamma_d$. 
\item $\delta_d$ is the smallest constant such that every $d$ selfadjoint contractions can be dilated to a $d$-tuple of commuting operators of the form $\delta_d \times$unitary.
\end{enumerate}
Tautologically, $\alpha_d \geq \gamma_d$, but since $A \prec B$ and $A = A^*$ imply that $A \prec \frac{1}{2}(B+B^*)$, we have that $\alpha_d \leq \gamma_d$ thus $\alpha_d = \gamma_d$. 
Likewise, $\alpha_d \geq \beta_d$ and $\gamma_d \geq \delta_d$, but by ``pushing the spectrum to the extreme" one can show that $\alpha_d = \beta_d$ and that $\gamma_d = \delta_d$ (see, e.g. \cite[Proposition 2.3]{PSS18}); thus, all these constants are equal. 

In a similar vein, one can show that if $x$ is a $d$-tuple of selfadjoint operators, then $c(x,u_0) := \min \{ c : x \prec cu_0 \}$ is equal to the minimal constant $c$ such that $x \prec c v$ for a tuple $v$ of commuting selfadjoint unitaries, and also equal to the minimal $c$ such that $x \prec c y$ for a commuting tuple $y$ of selfadjoint contractions. 

By \cite[Theorem 6.7]{PSS18} we know that $\alpha_d = \sqrt{d}$, thus, numerical experiments are not needed for shedding light on this constant. 
Moreover, the sharpness of $\sqrt{d}$ is exhibited already by a tuple of anti-commuting unitaries. 
In particular, for $d = 2$, we have that the universal commuting selfadjoint dilation constant is achieved already as $c(F,u_0) = \sqrt{2}$ for $F_1 = \begin{psmallmatrix} 1 & 0\\ 0 & -1 \end{psmallmatrix}$, $F_2 = \begin{psmallmatrix} 0 & 1\\ 1 & 0 \end{psmallmatrix}$.
The referee asked whether the dilation constants for pairs of selfadjoint random $N \times N$ unitaries converge. 
If $D^{(N)}$ is a sequence of diagonal selfadjoint unitary matrices with empirical distribution converging to $\frac{1}{2} \delta_{-1} + \frac{1}{2} \delta_1$, and if $U^{(N)}_1$ and $U^{(N)}_2$ are two independent random unitary matrices, and if we define $R_i^{(N)} = U^{(N)}_i D^{(N)} (U^{(N)}_i)^*$ for $i = 1, 2$, then it follows from \cite[Theorem 1.4]{CM14} that the pair of selfadjoint unitaries $R^{(N)} = (R_1^{(N)}, R_2^{(N)})$ is asymptotically free and converges strongly to a pair of free Bernoulli variables $r = (r_1, r_2)$. 
Now \cite[Theorem 3.1]{GSh21} shows that $\cW(R^{(N)}) \to \cW(r)$ levelwise so that Theorem \ref{thm:limit_c} applies and we find that 
\[
\liminf_{N\to\infty} c(R^{(N)},u_0) \geq  c(r,u_0). 
\]
To estimate $c(r,u_0)$, note that $r_1$ and $-r_2r_1r_2$ are freely independent Bernoulli variables, and therefore their sum $r_1 - r_2 r_1 r_2$ is distributed according to the arcsine distribution on $[-2,2]$ (see \cite[Example 12.8]{NiSp06}). 
It follows that the norm of the commutator $[r_1, r_2] = r_1 r_2 - r_2 r_1$ is $\|[r_1, r_2] \| = \|r_1 - r_2 r_1 r_2\| = 2$. 
Now the proof of \cite[Theorem 3.9]{GPSS21} (with $u = r_1$ and $v = r_2$) shows that $c(r,u_0) \geq \sqrt{2}$. 
On the other hand, $c(R^{(N)},u_0) \leq  \sqrt{2}$ because $\alpha_2 = \delta_2 = \sqrt{2}$, and so $\limsup_{N\to\infty} c(R^{(N)},u_0) \leq \sqrt{2}$.  
Putting the two inequalities together we conclude that 
\[
\lim_{N \to \infty}c(R^{(N)},u_0) = c(r,u_0) = \sqrt{2} ,
\]
as the referee suspected.

Another natural question to ask about the selfadjoint case is the following. We know that if $X^{(N)}$ is a $d$-tuple of (appropriately normalized) independent random Wigner $N \times N$ matrices satisfying some additional assumptions (in particular, if they are sampled from the GUE ensemble), then $X^{(N)}$ converges strongly to the tuple $s = (s_1, \ldots, s_d)$ of free semicircular operators. 
We know from \cite[Theorem 3.2]{GSh21} that $\cW(X^{(N)}) \to \cW(s)$. Therefore, by Theorem \ref{thm:limit_c}, we have that 
\[  
c(s,u_0)\leq \liminf_{N\to\infty} c(X^{(N)},u_0).
\]
However, we do not know much about $c(s,u_0)$. Rudimentary numerical experiments that we carried out (for values $N = 10, 50, 100, 500, 1000$) suggest that $\lim_{N\to\infty} c(X^{(N)}/2,u_0) = 1$ and combining with Theorem 2.1, this will imply that $c(s/2,u_0) = 1$ because it cannot be smaller than $1$.  
This would suggest the plausibility of the containment $\cW(s/2) \subsetneq \cW^{\operatorname{min}}([-1,1]^d)$ (compare with \cite[Theorem 5.4]{GSh21}, where it is established that $\cW^{\operatorname{min}}(\ol{\bB}_d) \subsetneq \cW(s/2)$, where $\ol{\bB}_d$ is the Euclidean unit ball in $\bR^d$). It would be interesting to prove this rigorously.
\end{remark}

\section{Computing dilation constants --- algorithmic aspects}\label{sec:algorithmic}

Our experiment consists of numerically approximating $c(\mathsf{U},u_0)$ for repeated samples of a pair of independent Haar unitaries $\mathsf{U} = (\mathsf{U}_1, \mathsf{U}_2)  = U^{(N)}$ for a growing sequence of values of $N$. Fixing $N$, we observe the distribution of the values of the dilation constant $c(\mathsf{U},u_0)$, and study its behavior as $N$ grows. 
The data that we collected is described in \Cref{sec:experimental_re}. In the remainder of this section, we explain how we approximate the dilation constants $c(\mathsf{U},u_0)$. 

\subsection{Approximation of the dilation constant}\label{subsec:approx}
To approximate $c(\mathsf{U},u_0)$ for a fixed given pair $\mathsf{U} = (\mathsf{U}_1, \mathsf{U}_2)$ of $N \times N$ unitary matrices, we first define a pair of commuting normal matrices $\mathsf{N} = (\mathsf{N}_1, \mathsf{N}_2)$ that approximates $u_0$ by letting the values of the diagonal of $\mathsf{N}$ run over all $k^2$ possible pairs of points in $V_k \times V_k$, where 
\[
V_k = \left\{ \exp\left(\frac{2\pi i m}{k}\right) : m=0,1,\ldots, k-1 \right\}. 
\]
The set $V_k$ consists of the extreme points of a polygon $P_k = \conv(V_k)$ in $\bC \cong \bR^2$. 
Therefore, $c(\mathsf{U},\mathsf{N})$ is the minimal constant $C$ such that $\mathsf{U}$ is a compression of a pair of commuting normal operators $\mathsf{T}$ with $\sigma(\mathsf{T}) \subset P_k \times P_k$. 
Note that since $P_k^2 \subset \ol{\bD}^2 \subset \cos(\pi/k)^{-1} \cdot P_k^2$, we find that 
\be\label{eq:within}
\cos(\pi/k) c(\mathsf{U},\mathsf{N}) \leq c(\mathsf{U},u_0) \leq c(\mathsf{U},\mathsf{N}) .
\ee
When $k = 20$, for example, we get an upper bound that lies within about $1.3 \%$ of the true value of $c(\mathsf{U},u_0)$, which, in turn, almost surely becomes an upper bound for $c(u_{\mathrm f},u_0)$ as $N \to \infty$, by \Cref{cor:c_converge}. 
It is worth recalling from \Cref{eq:free_bounds} that we already have the lower bound
\[
c(u_{\mathrm f},u_0) \geq \sqrt{2}. 
\]

\subsection{Algorithm for computing \texorpdfstring{$c(\mathsf{U},\mathsf{N})$}{c(U,N)}}\label{subsec:algo}
The following algorithm for computing the dilation constant $c(\mathsf{U},\mathsf{N})$ is an adaptation of the algorithm by Helton, Klep and McCullough \cite{HKM13} for determining whether there exists a UCP map between two sets of matrices. 
In this form it appeared first in \cite{ShaBlog}.

To numerically compute $c(\mathsf{U},\mathsf{N})$ for a fixed pairs of matrices, we set up the semidefinite program
\be\label{eq:obj}
\texttt{maximize } r
\ee
subject to the constraints that the $k^2$ matrix variables $\mathsf{C}_1, \ldots, \mathsf{C}_{k^2} \in M_N(\bC)$ satisfy
\be\label{eq:constraint1}
\mathsf{C}_j \geq 0, \quad \textrm{ for all } j = 1, \ldots, k^2, 
\ee
\be\label{eq:constraint2}
\sum_{j=1}^{k^2} \mathsf{C}_j = I_N
\ee
and
\be\label{eq:constraint3}
\sum_{j=1}^{k^2} (\mathsf{N}_i)_{jj} \mathsf{C}_j = r \mathsf{U}_i, \quad \textrm{ for } i=1,2.
\ee
Note that the constraints \cref*{eq:constraint1,eq:constraint2} mean that the map $E_{jj} \mapsto \mathsf{C}_j$ (for $j=1, \ldots, k^2$) defines a UCP map from the diagonal $k^2 \times k^2$ matrices to $M_N(\bC)$; by Arveson's extension theorem, this map extends to a UCP map $\Phi \colon M_{k^2}(\bC) \to M_N(\bC)$.  
The constraint \cref*{eq:constraint3} means that $\Phi(\mathsf{N}_i) = r\mathsf{U}_i$ for $i=1,2$. 
Thus, a solution to the above semidefinite program is the maximal $r$ such that there exists a UCP map $\Phi$ such that $\Phi(\mathsf{N}) = r \mathsf{U}$. 
The dilation constant $c(\mathsf{U},\mathsf{N})$ is therefore given by $c(\mathsf{U},\mathsf{N}) = r^{-1}$.

\section{Experimental results}\label{sec:experimental_re}

We now present the experimental results that support the hypothesis that 
\[
\lim_{N \to \infty} c(U^{(N)}, u_0) = \sqrt{2}
\]
almost surely. 
Recall that by \Cref{cor:c_converge}, the limit inferior $\liminf_{N \to \infty} c(U^{(N)}, u_0)$ is larger than the the constant $c(u_{\mathrm f}, u_0)$ which we seek. 
On the other hand, by \Cref{eq:free_bounds} we have $c(u_{\mathrm f}, u_0) \geq \sqrt{2}$, so if $\liminf_{N \to \infty} c(U^{(N)}, u_0) = \sqrt{2}$ a.s., then we would be able to deduce that $c(u_{\mathrm f}, u_0) = \sqrt{2}$ and also that $C_2 = c(u_{\mathrm u}, c_0) \leq 2 \sqrt{\frac{2}{3}} < 2$. 

We have even bigger confidence in a weaker claim, which is still stronger than the known upper bound. 
Note that all the constants $c(\mathsf{U}, \mathsf{N})$, to which we numerically compute approximations, are upper bounds for the values of $c(\mathsf{U},u_0)$, which can be thought of as samples of the random variables in question $c(U^{(N)}, u_0)$. 
No single instance that we computed fell above $\frac{\sqrt{2} + \sqrt{3}}{2} \approx 1.57$, which is the midpoint of the bounds given by \Cref{eq:free_bounds} for $d = 2$. We consider this to be strong evidence that $c(u_{\mathrm f}, u_0) < \frac{\sqrt{2} + \sqrt{3}}{2}$, showing that 
\[
  C_2 = c(u_{\mathrm u}, c_0) \leq c(u_{\mathrm u}, u_{\mathrm f}) c(u_{\mathrm f}, c_0) < \frac{2}{\sqrt{3}}\frac{\sqrt{2} + \sqrt{3}}{2}  =  1+\sqrt{{\tfrac{2}{3}}}< 2.
\]

\subsection{Experiment description and collected data}\label{subsec:exp_n_data}
We ran the algorithm described in \Cref{subsec:algo} on pairs $(\mathsf{U}, \mathsf{N})$ as in \Cref{subsec:approx} for various values of $N$ (the size of $\mathsf{U}$) and various values of $k$ (the size of $\mathsf{N}$). The number of variables in the problem has an order of magnitude
$
k^2 N^2
$
and likewise for the number of constraints. The complexity of the optimization algorithm used by SCS is quadratic in the size of the problem. Therefore, for certain values of $N$ and $k$, we limit the number of trials $m$ that we run in order to collect data, and settle for a few indicative runs.

We ran $m = 100$ experiments for $k = 20$ and each $N = 5 \cdot j$, for $j = 2, \ldots, 12$. For every choice of $N$ we plot the histogram (see \Cref{fig:histograms}). 
We also plot the mean and standard deviations as functions of $N$ (see \Cref{fig:combined-images-main}). 

\begin{figure}[ht]
    \centering
    \caption{Histograms of calculated dilation constants plotted separately for different matrix size $N$. As we expected, as $N$ grows, the values of the dilation constants accumulate closely to a value slightly above $\sqrt{2}$ (left dotted line) and below $\sqrt{2}/\cos(\pi/k)$ (right dotted line).}
    \label{fig:histograms}
    \includegraphics[width=\textwidth]{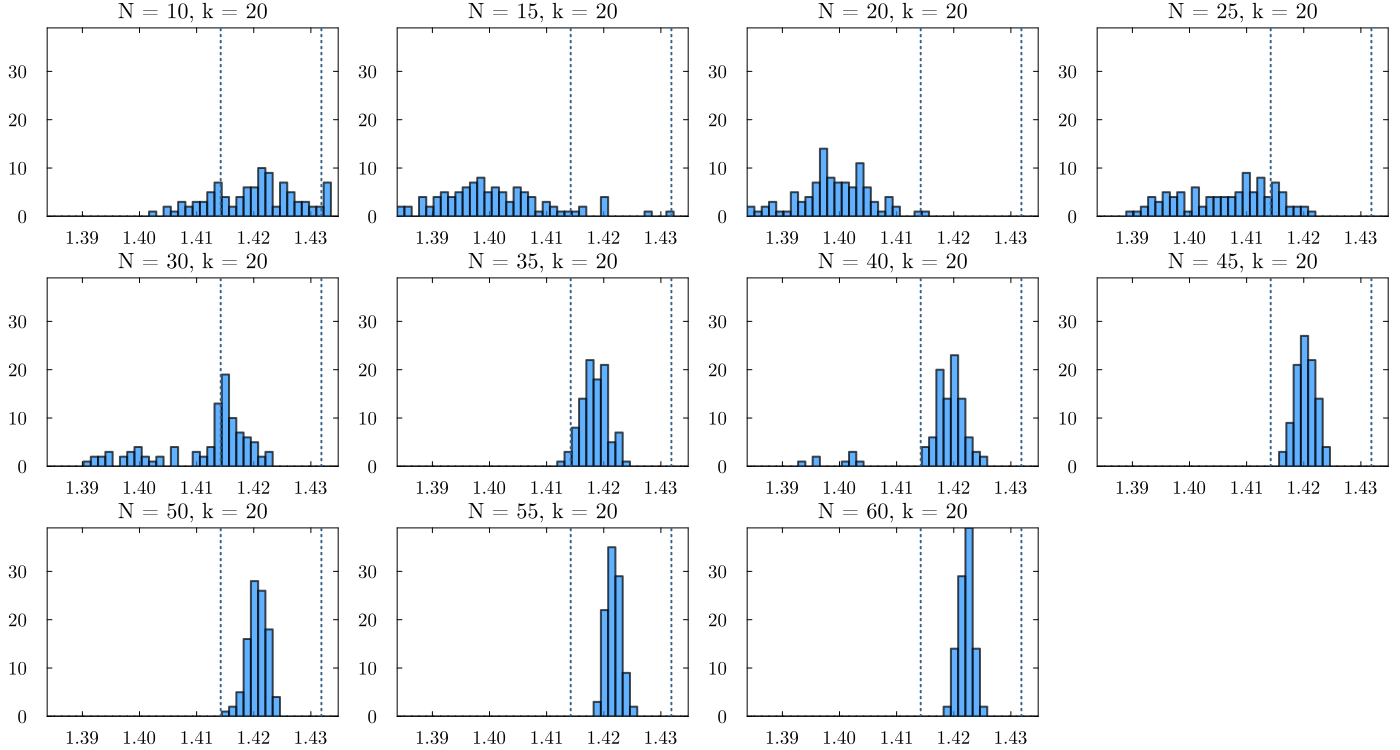}
\end{figure}
 
\begin{figure}[ht]
    \centering
    \caption{The mean and standard deviations of the values of dilations constants as functions of $N$.}
    \label{fig:combined-images-main}
    
    \begin{minipage}[c]{0.45\textwidth}
        \centering
        \includegraphics[width=\textwidth]{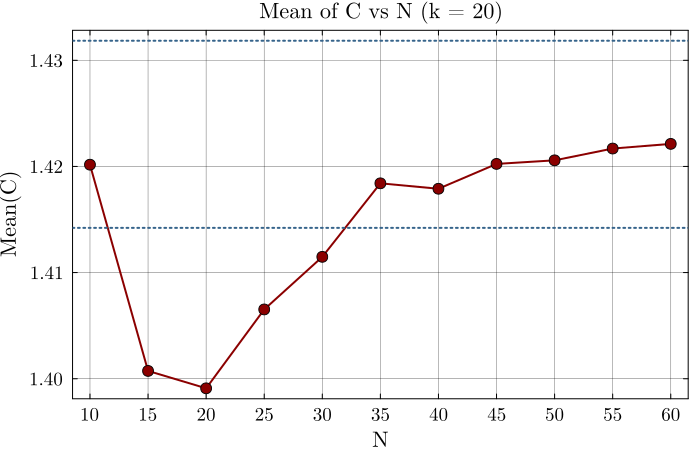}
    \end{minipage}
    \hfill
    \begin{minipage}[c]{0.45\textwidth}
        \centering
        \includegraphics[width=\textwidth]{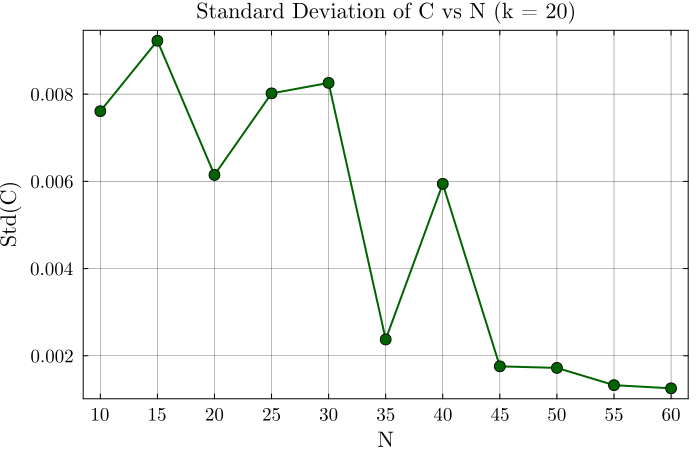}
    \end{minipage}
\end{figure}

We also ran a single experiment for $k = 20,30$ and every value of $N = 10, 25, 40, 55, 70, 85, 100$, to give a sense of the behavior of the true constant as $N$ grows large. Encouraged by the fact that standard deviations of computed dilation constants become small as $N$ grows, we hoped that perhaps the graphs of ``single shot'' experiments would already be indicative of convergence, but we did not observe any decisive trend in this range (see \Cref{fig:combined-images}). 
On the other hand, we do see that the calculated dilation constant remains below $\sqrt{2}/\cos(\pi/k)$, which is the correction to the hypothetical limit constant $\sqrt{2}$ due to approximating the unit circle by a regular $k$-gon. 

\begin{figure}[ht]
    \centering
    \caption{The values of the dilation constant calculated for a single sample of size $N \times N$ for $N = 10, 25, \ldots, 100$, with $k=20$ (left) and $k=30$ (right). Note that the values obtained for $N\geq 40$ are always below $\sqrt{2}/\cos(\pi/k)$ (upper dotted line).}
    \label{fig:combined-images}
    
    \begin{minipage}[c]{0.45\textwidth}
        \centering
        \includegraphics[width=\textwidth]{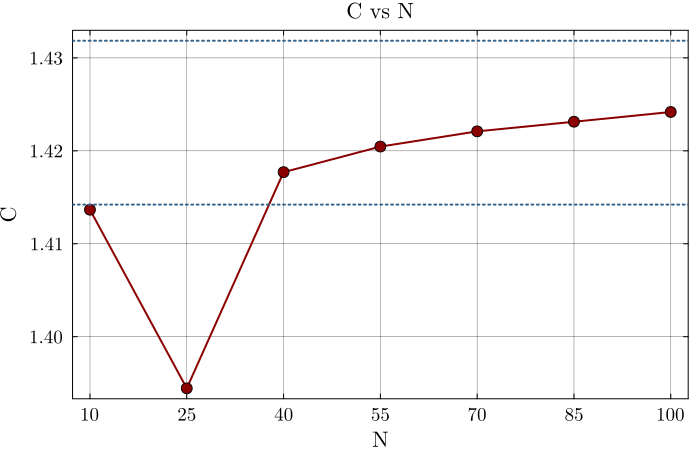}
    \end{minipage}
    \hfill
    \begin{minipage}[c]{0.45\textwidth}
        \centering
        \includegraphics[width=\textwidth]{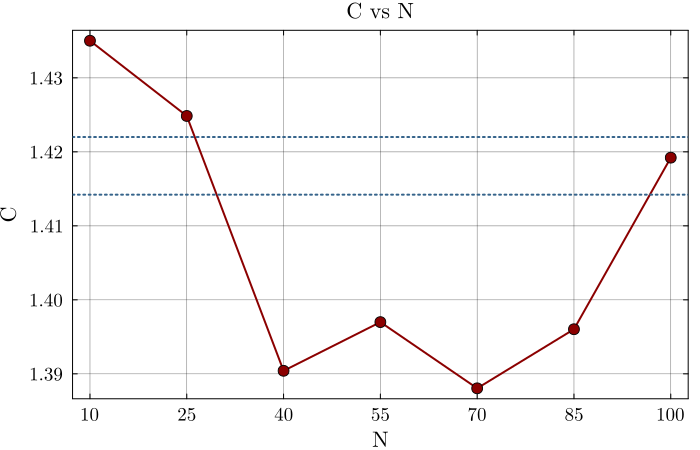}
    \end{minipage}
\end{figure}

To check larger $N$s, we ran single-shot experiments with $N = 25, 50, \ldots, 300$ $k=8$ (for which the error is about 8\%). The promising results are plotted in \Cref{fig:one-shot8}. Now the dilation constants appear to stabilize near $1.439$ before we reach the computational barrier. 
To verify the result, we computed the dilation constant for $N=125, k = 8$ and another 100 random pairs of unitaries; the histogram of the computed values, which cluster tightly around $1.439$, is shown in \Cref{fig:one-shot8}.

\begin{figure}[ht]
    \centering
    \caption{Dilation constants for a single sample of size $N \times N$ for $N = 25, 50, \ldots, 300$, with $k=8$ (left) alongside the histogram of 100 additional random trials for $N = 125$ (right). All values are well below $\sqrt{2}/\cos(\pi/k) \approx 1.53$.}
    \label{fig:one-shot8}
     \begin{minipage}[c]{0.45\textwidth}
        \centering
        \includegraphics[width=\textwidth]{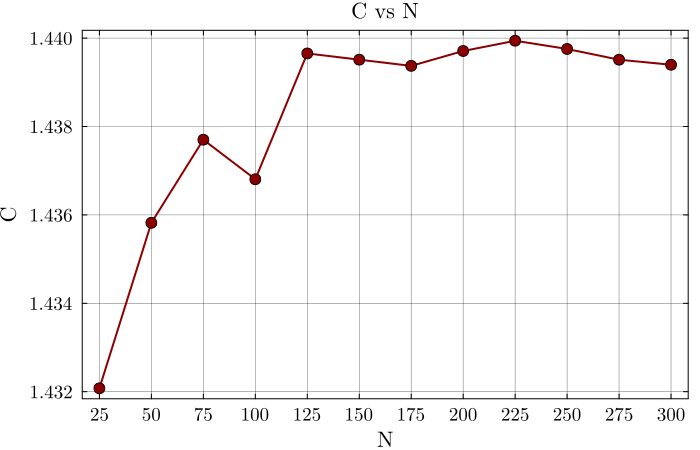}
    \end{minipage}
    \hfill
    \begin{minipage}[c]{0.45\textwidth}
        \centering
        \includegraphics[width=\textwidth]{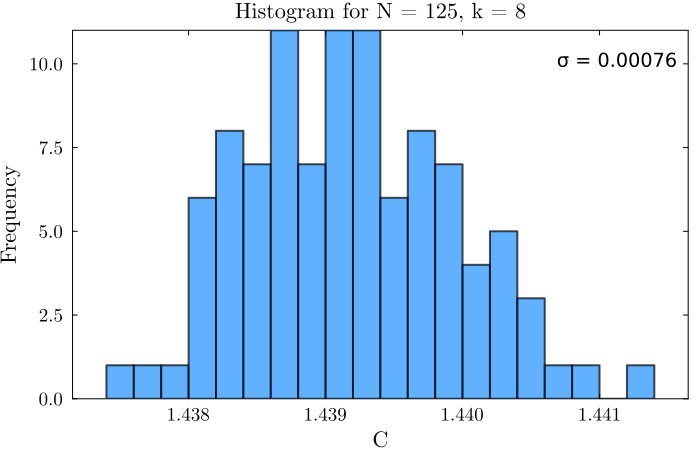}
    \end{minipage}    
\end{figure}

We repeated the experiment described on the left pane of \Cref{fig:one-shot8} ten more times. The results are shown in \Cref{fig:one-shot8_10}.  It seems that all paths are converging to something near $1.44$. Recall that by \Cref{thm:limit_c} and \cite[Theorem 3.1]{GSh21}, we expect to see
  \[
  c(u_{\mathrm f}, \mathsf{N})\leq\liminf c(U^{(N)},\mathsf{N})
  \]
almost every time, where $\mathsf{N}$ is the commuting pair with eigenvalues on the vertices of $P_k \times P_k$ for $k = 8$. 
The data suggests that the limit exists and is $\approx 1.44$. 

\begin{figure}[ht]
    \centering
    \caption{Ten random sequences $\{c(U^{(N)},\mathsf{N}) : N = 25, 50, \ldots, 300\}$.}
    \label{fig:one-shot8_10}
     \begin{minipage}[c]{0.5\textwidth}
        \centering
        \includegraphics[width=\textwidth]{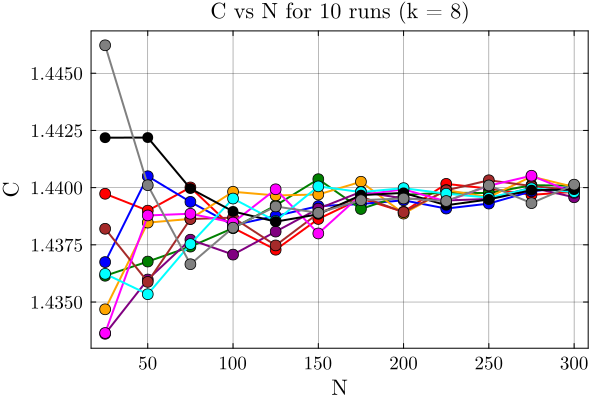}
    \end{minipage}
\end{figure}

\subsection{Code}
All numerical experiments were performed using the Julia programming language \cite{julia}. Optimization models were formulated using the Convex.jl package \cite{convexjl} and solved with the SCS solver \cite{scsjl}. We used DataFrames.jl \cite{dataframesjl} and CSV.jl \cite{csvjl} for data handling, and Plots.jl \cite{plotsjl} for visualization. 
Auxiliary scripts for running the series of experiments and visualizing the results were written in Visual Studio Code with the help of GitHub Copilot. 
The code was run on a MacBook Pro with Apple M4 Pro chip and 24 GB of memory. 
Our code is available at \url{https://oshalit.net.technion.ac.il/code/}.

\subsection{How reliable are the results}
Given that the results are based on numerical calculations of finitely many random samples $U^{(N)}$ for finitely many values of $N$ --- and not on a proof --- it is clear that we have not proved the conjectured $\lim_{N \to \infty} c(U^{(N)}, u_0) = \sqrt{2}$. 
The purpose of this short section is to address the subtler issue, of whether we can trust the result of the numerical computations that we carried out. 
Note that there is no question of trusting AI or not, since GitHub Copilot was only used for to help write simple scripts for running experiments and visualizing the results --- short, simple scripts whose correctness we verified ourselves; the issue is whether the {\em numerics} are reliable. 

The results that we report are consistent with the evidence that we have been observing for years when computing dilation constants for random matrices, using the optimization packages CVX in MATLAB and cvxpy in Python, invoking different solvers, and run on different computers and on the cloud, though we have not yet been able to reach values of $N$ and $k$ as large as we have here. 
It is likely that the SCS optimization package made the difference. 

The Splitting Conic Solver (SCS) \cite{ocpb:16,odonoghue:21} is a well established and tested solver. 
The optimization finished thousands of runs consistently with status ``solved''. It failed in a few cases when we fed the program problems that were too large, in which case no result was recorded. 

In order to increase our confidence in the results, we double-checked their reliability by computing the dilation constant for several pairs $(\mathsf{U},\mathsf{N})$ by the algorithm described in \Cref{subsec:algo} with both SCS as well as the commercial solver MOSEK \cite{mosekjl}. 
For several tests with $(N,k) = (30,14), (35,14), (40,10)$ and $(45,10)$ we obtained the same dilation constant up to an error of about $0.001$. 
This is reasonable given SCS's default tolerance $1\textup{e-}4$ and the fact that we only expect to approximate $c(u_{\mathrm f}, u_0)$ to within $0.01$ of its correct value for $k\leq 23$; see \Cref{eq:within}. 

The fact that two distinct packages, implementing different optimization algorithms, obtained similar values for the tested $N$ and $k$ is reassuring. MOSEK is considered reliable, but the high precision interior point method which it uses has complexity that scales cubically in the problem size, making it too memory intensive for large $N$ and $k$. Therefore, we could not run MOSEK for all the values of $(N,k)$ that we tested in the experiment.

Finally, the data documented here also fits well with the theory, in particular with our new results \Cref{thm:limit_c,cor:c_converge}, as with other rigorous bounds such as \Cref{eq:tri_ineq,eq:free_bounds}. 
The fact that the calculated dilation constants seem to accumulate just over the lower bound $\sqrt{2}$ we had for $c(u_{\mathrm f}, u_0)$ is another heuristic indication that the value $c(u_{\mathrm f}, u_0) = \sqrt{2}$ suggested by our experimentation is the correct value.

\appendix\section{A rigorous bound for the dilation constant \texorpdfstring{$C_2(n)$}{C2(n)}}

Denote by $C_2(n)$ the dilation constant such that for every pair of contractive $n\times n$ matrices, there exists a dilation to a pair of commuting matrices with norm less than or equal to $C_2(n)$. 
By Ando's theorem combined with the usual ``dilation theory in finite dimensions'' yoga \cite{hartz2021dilation,LevyShalit2014,McSh13} the constant $C_2(n)$ is also equal to:
\begin{enumerate}
\item the constant $C_{2,u}(n)$ such that for every pair of contractive $n\times n$ matrices, there exists a dilation to $C_{2,u}(n)$ times a pair of commuting unitary operators, and 
\item the constant $C_{2,u,\operatorname{fin}}(n)$ such that for every pair of contractive $n\times n$ matrices, there exists a dilation to $C_{2,u,\operatorname{fin}}(n)$ times a pair of commuting unitary {\em matrices}.
\end{enumerate}
To see this in detail, let $C$ be a constant, and suppose that a pair of $n \times n$ matrices $(A,B)$ has a dilation to $C(S,T)$ where $S,T$ are commuting contractive matrices. 
By Ando's theorem \cite{ando1963pair}, $(S,T)$ has a power dilation to a pair of commuting unitaries $(U,V)$, thus $(A,B)$ has a dilation to $C$ times a pair of commuting unitaries. This shows that $C_{2,u}(n) \leq C_2(n)$. 
Now, if a $(A,B)$ has a dilation to $C$ times a pair of commuting unitaries, then Theorem 1.5 in \cite{hartz2021dilation} tells us that $(A,B)$ has a dilation to $C$ times a pair $(U',V')$ of commuting unitary {\em matrices}. 
This shows that $C_{2,u,\operatorname{fin}}(n) \leq C_{2,u}(n)$. 
On the other hand, it is clear that $C_2(n) \leq C_{2,u,\operatorname{fin}}(n)$ and we conclude that 
\[
C_2(n) =  C_{2,u}(n) = C_{2,u,\operatorname{fin}}(n) . 
\]

\begin{theorem}
Let $A, B\in M_n(\mathbb{C})$ be unitary matrices. Then, for every $\lambda\in\mathbb{T}$, the matrices
  \[
    S_\lambda=
      \begin{pmatrix}
        A & \lambda B\\
        B &              A
      \end{pmatrix}
    , \ \ \ \ \
    T_\lambda=\begin{pmatrix}
    B                   &  A\\
    \bar\lambda A &  B
    \end{pmatrix}
  \]
commute, and 
 
  \[
    \inf_{\lambda\in\mathbb{T}}\textup{max}\bigl\{\|S_\lambda\|,\|T_\lambda\|\bigr\}\le\sqrt{2+2\sin\bigl(\tfrac{\pi}{2}(1-\tfrac{1}{n})\bigr)}.
  \]
\end{theorem}

\begin{proof}
Let $A, B$ and $\lambda$ be as above. Then $S_\lambda$ and $T_\lambda$ commute,
  \[
    S_\lambda T_\lambda=\begin{pmatrix} AB+BA & A^2+\lambda B^2\\ B^2+\bar\lambda A^2 & BA+AB\end{pmatrix}=T_\lambda S_\lambda.
  \]
Next, we estimate the norm of $S_\lambda$. Note that
  \[
    S_\lambda^*S_\lambda=\begin{pmatrix} 2 & \lambda A^*B+B^*A\\ \bar\lambda B^*A+A^*B & 2\end{pmatrix},  
  \]
since $A$ and $B$ are unitary. Thus,
  \[    
    \|S_\lambda\|\le\sqrt{2+\|\lambda A^*B+B^*A}\|=\sqrt{2+\|\lambda+(B^*A)^2\|}.
  \]
This also provides an upper bound for $\|T_\lambda\|$, since
  \[
    T_\lambda\begin{pmatrix} 0 & \lambda\\ 1 & 0\end{pmatrix}=\begin{pmatrix} A & \lambda B\\ B & A\end{pmatrix}=S_\lambda.
  \]
Because $(B^*A)^2$ is unitary, we have that
  \[
    \|\lambda+(B^*A)^2\|=\sup_{\mu\in\sigma((B^*A)^2)}|\lambda+\mu|.
  \]
Next note that since $\sigma\bigl((B^*A)^2\bigr)$ contains at most $n$ distinct points, there exists an interval $[a,b]\subset\mathbb{R}$ such that:
  \begin{enumerate}[label=\roman*)]
    \item $\sigma\bigl((B^*A)^2\bigr)\subset\bigl\{e^{i\theta},\ \theta\in[a,b]\bigr\}$,
    \item $|a-b|\le 2\pi(1-1/n)$.
  \end{enumerate}
Let $c=\tfrac{b-a}{2}$ and set $\lambda=-e^{i\tfrac{b+a}{2}}$. Then
    \begin{align*}
      \sup_{\mu\in\sigma(B^*A)^2}|\lambda+\mu|
                                                \le\sup_{\theta\in[a,b]}\left|-e^{i\tfrac{b+a}{2}}+e^{i\theta}\right|
        =\sup_{\theta\in[a,b]}\left|e^{i(\tfrac{b+a}{2}-\theta)}-1\right|
       &
        =\sup_{\theta\in[-c,c]}|e^{i\theta}-1| \\
       &
        =\sup_{\theta\in[-c,c]}2\bigl|\sin(\theta/2)\bigr|.
    \end{align*}
Since $[-c,c]\subset[-\pi,\pi]$ and $|\sin(t)|$, restricted to $[-\tfrac{\pi}{2},\tfrac{\pi}{2}]$, is increasing in $|t|$, we obtain that 
  \[
    \sup_{\theta\in[-c,c]}2\bigl|\sin(\theta/2)\bigr|\le2\sin(\tfrac{c}{2})\le2\sin\left(\tfrac{\pi}{2}(1-1/n)\right).
  \]
\end{proof}

\begin{corollary}\label{cor:C2n}
It holds that
  \[
    C_2(n)\le\sqrt{2+2\sin\bigl(\tfrac{\pi}{2}(1-\tfrac{1}{2n})\bigr)} < 2.
  \]
\end{corollary}

\begin{proof}
Observe that every pair of contractive $n\times n$ matrices $(A,B)$ has a dilation to a pair of unitary $2n\times2n$ matrices via
  \[
    \left(\begin{pmatrix} A & \sqrt{1-AA^*} \\ \sqrt{1-A^*A} & -A^* \end{pmatrix}, \begin{pmatrix} B  & \sqrt{1-BB^*}\\ \sqrt{1-B^*B} & -B^* \end{pmatrix}\right).
  \]
Thus the corollary follows from the previous theorem. 
\end{proof}

\noindent {\bf Acknowledgements.} The authors would like to thank the referee for suggesting several corrections and improvements. In particular, Remark \ref{rem:selfadjoint} was stimulated by the referee report.


\bibliographystyle{myalphaurl-sortbyauthor}
\bibliography{maltebib}
\setlength{\parindent}{0pt}
\end{document}